\documentclass[11pt]{amsart}
\usepackage{amsfonts}
\usepackage{}
\usepackage{mathrsfs}
\usepackage{bbm}
\usepackage{amssymb}
\usepackage{indentfirst}
\usepackage{latexsym,amsfonts,amssymb,amsmath}
\newtheorem{theorem}{Theorem}[section]
\newtheorem{lemma}[theorem]{Lemma}
\newtheorem{corollary}[theorem]{Corollary}
\newtheorem{question}[theorem]{Question}
\theoremstyle{definition}
\newtheorem{definition}[theorem]{Definition}
\newtheorem{proposition}[theorem]{Proposition}
\theoremstyle{remark}
\newtheorem{remark}[theorem]{Remark}

\catcode`\@=12

\def\R{{\mathbb R}}

\begin{document}

\title
{The continuous $d$-open homomorphism images and subgroups of $\mathbb{R}$-factorizabile paratopological groups }

\author{Li-Hong Xie}\thanks{}
\address{(L.H. Xie) School of Mathematics and Computational Science, Wuyi University, Jiangmen 529020, P.R. China} \email{xielihong2011@aliyun.com; yunli198282@126.com}
\author{Peng-Fei Yan}\thanks{}
\address{(P.F. Yan) School of Mathematics and Computational Science, Wuyi University, Jiangmen 529020, P.R. China} \email{ypengfei@sina.com}

\subjclass[2010]{54B20, 54D20}
\keywords{$\mathbb{R}$-factorizable paratopological groups; $z$-embedded; $d$-open mappings}
\thanks{This research is supported by NSFC (Nos. 11601393, 11861018) and the Innovation Project of Department of Education of Guangdong Province(no:2018KTSCX231).
}

\begin{abstract}
In this paper, we prove that: (1) Let $f:G\rightarrow H$ be a continuous $d$-open surjective homomorphism; if $G$ is an $\mathbb{R}$-factorizabile paratopological group, then so is $H$. Peng and Zhang's result \cite[Theorem 1.7]{PZ} is improved. (2) Let $G$ be a regular $\mathbb{R}$-factorizable paratopological group; then every subgroup $H$ of $G$ is $\mathbb{R}$-factorizable  if and only if $H$ is $z$-embedded in $G$. This result gives out a positive answer to an question of M.~Sanchis and M.~Tkachenko \cite[Problem 5.3]{ST}.
\end{abstract}

\maketitle

\section{Introduction}
A \textit{paratopological} group is a group
with a topology such that multiplication on the group is jointly
continuous. If in addition inversion on the group is continuous, then it is
called a \textit{topological} group.

For every continuous real-valued function $f$ on a compact topological
group $G$, one can find a continuous homomorphism $p\colon G\to L$
onto a second countable topological group $L$  and a continuous real-valued
function $h$ on $L$ such that $f=h\circ p$ (see \cite[Example 37]{Pon}).
The conclusion remains valid for pseudocompact topological groups,
a result due to W.\,W.~Comfort and K.\,A.~Ross \cite{CR}. These facts
motivated M.~Tkachenko to introduce \textit{$\R$-factorizable}
groups in \cite{Tk} as the topological groups $G$ with the property that
every continuous real-valued function on $G$ can be factorized through
a continuous homomorphism onto a second countable topological
group. The class of $\R$-factorizable groups is unexpectedly wide.
For example, it contains arbitrary subgroups of $\sigma$-compact
(and even Lindel\"{o}f  $\Sigma$-) groups, topological products of
Lindel\"{o}f $\Sigma$-groups, and their dense subgroups \cite{Tk2}.
For other properties of this class of topological groups, the reader is
referred to \cite{Tk4,Tk3}. Similarly to the case of topological groups, M.~Sanchis and
M.~Tkachenko introduced in \cite{ST} the classes of $\R_i$-factorizable
paratopological groups, for $i\in\{1,2,3\}$. The need in
the use of four different subscripts was due to the fact that the classes
of $T_1$, Hausdorff and regular
paratopological groups are all distinct, while $T_0$ topological groups
are completely regular.

\begin{definition}\cite{ST,XTS}\label{Def1}
\textit{A paratopological group $G$ is $\R_0$-factorizable
$(\,\R_i$-factorizable, for $i = 1, 2, 3, 3.5)$ if for every continuous
real-valued function $f$ on $G$, one can find a continuous
homomorphism $\pi\colon G \to H$ onto a second-countable
paratopological group $H$ satisfying the $T_0$ (resp., $T_i+T_1$)
separation axiom and a continuous real-valued function $h$ on
$H$ such that $f = h\circ\pi$. If we do not impose any separation
restriction on $H$, we obtain the concept of $\R$-factorizability.}
\end{definition}

The following result shows that all $\R$-factorizable and
$\R_i$-factorizable, (for $i =0, 1, 2, 3)$ paratopological groups coincide.
\begin{theorem}\cite[Theorem 3.8]{XTS}\label{T1}
 Every $\mathbb{R}$-factorizable paratopological group is $\mathbb{R}_3$-factorizable. Hence the concepts of $\mathbb{R}$-, $\mathbb{R}_0$-,
$\mathbb{R}_1$- ,$\mathbb{R}_2$-, and $\mathbb{R}_3$-factorizability coincide in the class of paratopological groups.
\end{theorem}

It is well known that every quotient group of $\R$-factorizable topological group is $\R$-factorizable \cite[Theorem 3.10]{Tk}. This motivated M.Sanchis and M.~Tkachenko posed the following question:

\begin{question}\cite[Problem 5.2] {ST}
Let $G$ be an $\mathbb{R}_{i}$-factorizable paratopological group,
for some $i\in\{1, 2, 3, 3.5 \}$. Is every open continuous homomorphic image $H$ of
$G$ an
$\mathbb{R}_{i}$-factorizable paratopological group, provided that $H$ satisfies the $T_{i}$-separation axiom?
\end{question}

To solve this problem, the first author and S. Lin \cite{XS} introduce the concept of property $\omega$-$QU$ in paratopological groups (see Definition \ref{D}). They give this question a partial answer. In 2013, the first author, S. Lin and M.~Tkachenko proved that: Every quotient of a totally $\omega$-narrow (or Tychonoff)$\R$-factorizable paratopological group $G$ is $\R$-factorizable \cite[Propositions 3.19 and 3.20]{XTS}. And in the same paper, they point out that it is still an open problem whether every open continuous homomorphic image of an $\mathbb{R}$-factorizable paratopological
group is $\mathbb{R}$-factorizable. Recently, L.X. Peng and P. Zhang  answered this problem affirmatively (see \cite[Theorem 1.7]{PZ}).

Let $f: X \rightarrow Y$ be a mapping. Then $f$ is said to be
{\it $d$-open} if for any open set $O$ of $X$ there exists an open set $V$ of $f(X)$ such that $f(O)$ is a
dense subset of $V$. Clearly, every continuous open mapping is continuous $d$-open. In this paper, firstly we show that every continuous $d$-open homomorphism preserves the $\mathbb{R}$-factorizablilties in paratopological groups.

Secondly. Recall that a subspace $Y$ of a space $X$ is said to be {\it $z$-embedded} in $X$ if for every zero-set $Z$ in $Y$, there exists a zero-set $C$ in $X$ such that $Z = C\cap Y$. It was shown in \cite[Theorem 2.4]{HT} that a subgroup $G$ of an $\R$ -factorizable topological group $H$ is
$\R$-factorizable if and only if $G$ is $z$-embedded in $H$. In \cite{ST}, M.~Sanchis and
M.~Tkachenko proved that a $z$-embedded subgroup of an $\mathbb{R}_i$-factorizable paratopological group is $\mathbb{R}_i$-factorizable, for $i = 1, 2, 3$ \cite[Theorem 3.12]{ST} and posed the following question:

\begin{question}\cite[Problem 5.3]{ST}\label{Q}
 Let $G$ be a subgroup of a completely regular paratopological group $H$. If $G$ is $\R_3$-factorizable, must it be
$z$-embedded in $H$? What if $G$ is totally Lindel\"{o}f?
\end{question}

In this paper, we give a positive answer to the above question (see Theorem \ref{TT}).

Let $X$ be a space with a topology $\tau$. Then the family $\{\text{~Int~} \overline{U} : U\in \tau\}$ constitutes a base for a coarser topology $\sigma$ on $X$. The space $X_{sr} = (X, \sigma)$ is called the {\it semiregularization}
of $X$. The following very useful result was proved by Ravsky in \cite{Rav01} (see also \cite[Theorem 2.2]{Tk_inf}):

\begin{theorem}\label{th1}
Let $G$ be an arbitrary paratopological group. Then the space $G_{sr}$ carrying the same group
structure is a $T_3$ paratopological group. If $G$ is Hausdorff, then $G_{sr}$ is a regular paratopological group.
\end{theorem}

The spaces we consider are not assumed to satisfy any separation
axiom, unless the otherwise is stated explicitly. Further, $T_{3}$ and
$T_{3.5}$ do not include $T_{1}$, while \lq{regular\rq} and \lq{completely
regular\rq} mean $T_3 + T_1$ and $T_{3.5} + T_1$, respectively.

\section{Main results}
In this section, firstly, we prove that every continuous $d$-open homomorphism preserves the $\mathbb{R}$-factorizablilties in paratopological groups.

\begin{theorem}\label{TTT}
Let $f:G\rightarrow H$ be a continuous $d$-open surjective homomorphism. If $G$ is $\mathbb{R}$-factorizabile paratopological group, then so is $H$
\end{theorem}

To prove Theorem \ref{TTT}, we need establish some facts as following.

A real-valued function $f$ on a paratopological group $G$ is {\it left}
(resp. {\it right}) {\it $\omega$-quasi-uniformly continuous} if, for
every $\varepsilon>0$, there exists a countable family $\mathcal {U}$ of open neighbourhoods of the identity in $G$ such that for every $x\in G$,
there exists $U\in \mathcal {U}$ such that $|f(x)-f(y)|<\varepsilon$
whenever $x^{-1}y\in U$ (resp. whenever $yx^{-1}\in U$) \cite[Definition 4.1]{XS}. A real-valued function $f$ on a paratopological group is {\it
$\omega$-quasi-uniformly continuous} if $f$ is both left and right
$\omega$-quasi-uniformly continuous \cite[Definition 4.2]{XS}.

\begin{definition}\cite[Definition 4.7]{XS}\label{D}
A paratopological group $G$ has {\it property $\omega$-$QU$} if each continuous
real-valued function on $G$ is $\omega$-quasi-uniformly continuous.
\end{definition}

\begin{remark}
One can easily show that a paratopological group $G$ has property $\omega$-$QU$ if and only if for each continuous function $f:G\rightarrow \mathbb{R}$ and any $\epsilon>0$, there is a family $\{V_n:n\in \omega\}$ of open neighbourhoods of the identity in $G$ such that for each $g\in G$, there is $n\in \omega$ satisfying $f(gV_n)\subseteq (f(g)-\epsilon, f(g)+\epsilon)$ and $f(V_ng)\subseteq (f(g)-\epsilon, f(g)+\epsilon)$.
\end{remark}

\begin{proposition} \label{p1}
Let $f:G\rightarrow H$ be a continuous $d$-open surjective homomorphism. If $G$ has property $\omega$-$QU$, then so is $H$
\end{proposition}

\begin{proof}
Let $k:H\rightarrow \mathbb{R}$ be a continuous function. Fix $\epsilon>0$. Since $G$ has property $\omega$-$QU$, there is a family $\{V_n:n\in \omega\}$ such that for each $g\in G$ there is $n_0\in\omega$ satisfying that $k\circ f(gV_{n_0})\subseteq (k\circ f(g)-\frac{\epsilon}{2}, k\circ f(g)+\frac{\epsilon}{2})$ and $k\circ f(V_{n_0}g)\subseteq (k\circ f(g)-\frac{\epsilon}{2}, k\circ f(g)+\frac{\epsilon}{2})$. Since $f$ is a $d$-open homomorphism, the family $\{\overline{f(V_n)}^\circ: n\in \omega\}$ is the open neighbourhoods of the identity in $H$. Now we shall show that the family $\{\overline{f(V_n)}^\circ: n\in \omega\}$ satisfying that for each $h\in H$ there is $i\in \omega$ such that $ k(h\overline{f(V_{i})}^\circ)\subseteq (k(h)-\epsilon, k(h)+\epsilon)$ and $k(\overline{f(V_{i})}^\circ h)\subseteq (k(h)-\epsilon, k(h)+\epsilon)$. This implies that $H$ has the property $\omega$-$QU$.

 Take a $g\in G$ such that $h=f(g)$. Then there is $n_0\in\omega$ satisfying that $$k\circ f(gV_{n_0})\subseteq (k\circ f(g)-\frac{\epsilon}{2}, k\circ f(g)+\frac{\epsilon}{2})$$ and $$k\circ f(V_{n_0}g)\subseteq (k\circ f(g)-\frac{\epsilon}{2}, k\circ f(g)+\frac{\epsilon}{2}).$$
 Hence $$k(h\overline{f(V_{n_0})}^\circ)=k(\overline{hf(V_{n_0})}^\circ)=k(\overline{f(gV_{n_0})}^\circ)\subseteq k(\overline{f(gV_{n_0})})\subseteq \overline{k(f(gV_{n_0}))}$$$$\subseteq\overline{(k\circ f(g)-\frac{\epsilon}{2}, k\circ f(g)+\frac{\epsilon}{2})}=[k(h)-\frac{\epsilon}{2}, k(h)+\frac{\epsilon}{2}]\subseteq (k(h)-\epsilon, k(h)+\epsilon) $$ and
 $$k(\overline{f(V_{n_0})}^\circ h)=k(\overline{f(V_{n_0})h}^\circ )=k(\overline{f(V_{n_0}g)}^\circ)\subseteq k(\overline{f(V_{n_0}g)})\subseteq \overline{k(f(V_{n_0}g))}$$$$\subseteq\overline{(k\circ f(g)-\frac{\epsilon}{2}, k\circ f(g)+\frac{\epsilon}{2})}=[k(h)-\frac{\epsilon}{2}, k(h)+\frac{\epsilon}{2}]\subseteq (k(h)-\epsilon, k(h)+\epsilon). $$
\end{proof}

\begin{lemma}\label{L1}
Let $f_1:X\rightarrow Y$, $g_1:X\rightarrow X'$ $g_2:Y\rightarrow Y'$ and $f_2:X'\rightarrow Y'$ be continuous surjective mappings such that $g_2\circ f_1=f_2\circ g_1$. If $f_1$ and $g_2$ are $d$-open, then so is $f_2$.
\end{lemma}

\begin{proof}
Take any open set $V$ in $X'$. Then $g_1^{-1}(V)$ is open in $X$. Since $f_1$ is $d$-open, we have $f_1(g_1^{-1}(V))\subseteq \overline{f_1(g_1^{-1}(V))}^\circ$. Since $g_2\circ f_1=f_2\circ g_1$ and $g_2$ is $d$-open, we have $$f_2(V)=g_2(f_1(g_1^{-1}(V)))\subseteq g_2(\overline{f_1(g_1^{-1}(V))}^\circ)\subseteq \overline{g_2(\overline{f_1(g_1^{-1}(V))}^\circ)}^\circ$$ $$\subseteq \overline{g_2(\overline{f_1(g_1^{-1}(V))})}\subseteq \overline{\overline{g_2(f_1(g_1^{-1}(V)))}}=\overline{g_2(f_1(g_1^{-1}(V)))}=\overline{f_2(V)}.$$

Hence we have shown that $f_2(V)\subseteq \overline{g_2(\overline{f_1(g_1^{-1}(V))}^\circ}^\circ)\subseteq \overline{f_2(V)}$, which implies that $f_2$ is $d$-open.

\end{proof}

The following result was proved in \cite[Lemma 4.2]{Xie}. For the sake of completeness we give out the proof.
\begin{lemma}\label{lema2}
Let $X$ be a space and $X_{sr}$ a semiregularization of $X$. Then the identity mapping $i:X\rightarrow X_{sr}$ is a continuous $d$-open mapping.
\end{lemma}

\begin{proof}
Clearly, the mapping $i$ is continuous. Take any open set $U$ of $X$. Then one can easily show that $i(U)$ dense in $\text{~Int~} \overline{U}$ which is open in $X_{sr}$. This implies that $i$ is a $d$-open mapping. Indeed, noting that $\overline{U}\subseteq \overline{i(U)}$, we have that $i(U)=U\subseteq\text{~Int~} \overline{U}\subseteq \overline{i(U)}$.
\end{proof}

Let $\varphi_ {G, 2} : G \rightarrow H$ be a continuous surjective homomorphism of semitopological groups. The pair $(H, \varphi_{G, 2})$
is called a $T_2$ -reflection of $G$ if $H$ is a Hausdorff semitopological group and for every continuous mapping
$f : G \rightarrow X$ of $G$ to a Hausdorff space $X$, there exists a continuous mapping $h  : H \rightarrow X$ such that
$f = h\circ\varphi_{ G, 2}$. We denote by $T_2(G)$ the $T_2$-reflection of $G$, thus omitting the corresponding homomorphism
$\varphi_{ G, 2}$ \cite{T2}. The mapping $\varphi_{ G, 2}$ is called the canonical homomorphism of $G$ onto $T_2(G)$.

\begin{proposition}\label{p2}
Let $f:G\rightarrow H$ be a continuous surjective homomorphism of paratopological groups. $G_{sr}$ and $H_{sr}$ are semiregularizations of $G$ and $H$, respectively. If $f$ is $d$-open, then

\begin{enumerate}
\item[(1)] if $H$ is Hausdorff, then $f:G_{sr}\rightarrow H_{sr}$ is a continuous $d$-open mapping;
\item[(2)] there is a continuous $d$-open surjective homomorphism $\tilde{f}:T_2(G)\rightarrow T_2(H)$ such that $\tilde{f}\circ\varphi_{G,2}=\varphi_{H,2}\circ f.$
\end{enumerate}
\end{proposition}

\begin{proof}
(1) According to Lemmas \ref{L1} and \ref{lema2} it is enough to show that $f:G_{sr}\rightarrow H_{sr}$ is continuous. Take any $x\in G$ and any open neighbourhood $U$ of $f(x)$ in $H_{sr}$. Without loss of generality, we can assume that $U$ is a regular open set in $H$, i.e., $\overline{U}^\circ=U$. Since $H$ is a Hausdorff paratopological group, according to Theorem \ref{th1} there is a regular open neighbourhood $V$ of $f(x)$ such that the closure of $V$ in $H$ is contained in $U$. Noting that $f:G\rightarrow H$ is continuous, there is an open neighbourhood $W$ of $x$ in $G$ such that $f(W)\subseteq V$. Clearly, the set $\overline{W}^\circ$ is regular open in $G$ and contains $x$. Now we shall show that $f(\overline{W}^\circ)\subseteq U$, which implies that $f:G_{sr}\rightarrow H_{sr}$ is continuous.

In fact, this directly follows from the following fact $$f(\overline{W}^\circ)\subseteq f(\overline{W})\subseteq \overline{f(W)}\subseteq \overline{V}\subseteq U.$$

(2) Since $T_2(G)$ is the $T_2$-reflection of $G$ and $T_2(H)$ is Hausdorff, there is a continuous function $\tilde{f}:T_2(G)\rightarrow T_2(H)$ such that $\tilde{f}\circ\varphi_{G,2}=\varphi_{H,2}\circ f.$ Noting that $\varphi_{G,2},$ $f$ and $\varphi_{H,2}$ are surjective homomorphism, one can easily show that $\tilde{f}$ is also a surjective homomorphism. It is well known that $\varphi_{H,2}$ are continuous open, so $\varphi_{H,2}$ is continuous $d$-open. Hence, it follows from Lemma \ref{L1} that $\tilde{f}:T_2(G)\rightarrow T_2(H)$ is $d$-open.
\end{proof}

{\bf Proof of Theorem \ref{TTT} }
Since a paratopological group $Q$ is $\mathbb{R}$-factorizable iff so is $T_2(Q)$ \cite[Lemma 1.3]{PZ}, by (2) of Proposition \ref{p2} we can assume that $G$ and $H$ are Hausdorff. Let $G_{sr}$ and $H_{sr}$ are semiregularizations of $G$ and $H$, respectively. Then by (1) of Proposition \ref{p2} we have that $f:G_{sr}\rightarrow H_{sr}$ be a continuous $d$-open surjective homomorphism. Since a paratopological group $F$ is $\mathbb{R}$-factorizable iff so is $F_{sr}$ \cite[Lemma 1.1]{PZ}, we only show that $H_{sr}$ is $\mathbb{R}$-factorizable.

It is well known that every regular paratopological group is completely regular \cite[Corollary 5]{BR}. Thus according to Theorem \ref{th1} we have $G_{sr}$ and $H_{sr}$ are Tychonoff. Thus, by \cite[Lemma 1.1]{PZ}, $G_{sr}$ is a Tychonoff $\mathbb{R}$-factorizable paratopological group. Since a Tychonoff paratopological group is $\mathbb{R}$-factorizable iff it is totally $\omega$-narrow and has property $\omega$-$QU$ \cite[Theorem 3.21]{XTS}, $G_{sr}$ is totally $\omega$-narrow and has property $\omega$-$QU$. Observing that $f:G_{sr}\rightarrow H_{sr}$ be a continuous $d$-open surjective homomorphism and a continuous homomorphic image of a totally $\omega$-narrow paratopological group is totally $\omega$-narrow \cite[Proposition 3.4]{ST1}, by Proposition \ref{p1}, $H_{sr}$ is a Tychonoff totally $\omega$-narrow paratopological group with property $\omega$-$QU$. Thus $H_{sr}$ is $\mathbb{R}$-factorizable by \cite[Theorem 3.21]{XTS}. This completes the proof.

\begin{corollary}\cite[Theorem 1.7]{PZ}
Let $H$ be a paratopological group. If $H$ is a continuous open homomorphic image of an
$\R$-factorizable paratopological group, then $H$ is $\R$-factorizable.
\end{corollary}

Next, we shall give an positive answer to Question \ref{Q}. Recall that a paratopological group $G$ is {\it $\omega$-narrow} if for each open neighbourhood of the identity in $G$ there is a countable subset $A$ of $G$ such that $AU=UA=G$. A paratopological group $G$ is called {\it totally $\omega$-narrow} if $G$ is a continuous homomorphism image of an $\omega$-narrow topological group. It is well known that every regular totally $\omega$-narrow paratopological group $G$ has countable index of regularity, i.e., $Ir(G)\leq \omega$ \cite[Theorem 2]{IS}, and evrey totally $\omega$-narrow first countble paratopological group has a countable base \cite[Proposition 3.5]{ST1}. Since a continuous homomorphic image of a totally $\omega$-narrow paratopological group is totally $\omega$-narrow \cite[Proposition 3.4]{ST1}, form the proof of \cite[Theorem 3.6]{T} one can obtain the following result:

\begin{lemma}\label{LL1}
Let $G$ be a regular totally $\omega$-narrow paratopological group. Then for each open neighbourhood $U$ of the identity $e$ in $G$, there are a continuous homomorphism $p_U: G \rightarrow H$ onto a second-countable regular paratopological group $H$ and a neighbourhood $V$ of the
neutral element in $H$ such that $p_U^{-1}(V)\subseteq U$.
\end{lemma}

\begin{proposition}\label{P}
Let $G$ be a regular totally $\omega$-narrow paratopological group. Then every $\mathbb{R}$-factorizable subgroup $H$ of $G$ is $z$-embedded in $G$.
\end{proposition}

\begin{proof}
Let $Z$ be a zero-set in $H$. Then there is a continuous function $f:H\rightarrow \mathbb{R}$ such that $Z=f^{-1}(0)$. Since $H$ is $\mathbb{R}$-factorizable, there are a continuous homomorphism $h$ on $H$ onto a regular second paratopological group $K$, and a continuous function $g:K\rightarrow \mathbb{R} $ such that $f=g\circ h$. Let $\{O_n:n\in \omega\}$ be a base at the identity in $K$. Then the family $\{h^{-1}(O_n):n\in \omega\}$ is a countable open neighbourhood of the identity in $H$. Since $G$ is regular totally $\omega$-narrow, by Lemma \ref{LL1}, for each $n\in \omega$ there are a continuous homomorphism $p_{O_n}: G \rightarrow F_{O_n}$ onto a second-countable regular paratopological group $F_{O_n}$ and a neighbourhood $V_{O_n}$ of the neutral element in $F_{O_n}$ such that $p_{O_n}^{-1}(V_{O_n})\cap H\subseteq h^{-1}(O_n)$. Let $p$ be the diagonal product of the family $\{p_{O_n}:n\in \omega\}$. Then $p(G)$ is a metrizable paratopological group. Since $p_{O_n}^{-1}(V_{O_n})\cap H\subseteq h^{-1}(O_n)$ and $\{O_n:n\in \omega\}$ be a base at the identity in $K$, we have that $\text{ker}p\cap H\subseteq \text{ker}h$, where $\text{ker}p$ and $\text{ker}h$ are the homomorphism kernels of $p$ and $h$, respectively. Hence there is a natural homomorphism $j:p(H)\rightarrow K$ such that $h=j\circ p|_H$.

 Now we shall show that the homomorphism $j:p(H)\rightarrow K$ is continuous. In fact, it is enough to show that $j$ is continuous at the identity of $p(H)$. Take any open neighbourhood $W$ of the identity in $K$. Since $\{O_n:n\in \omega\}$ is a base at the identity in $K$, there is $O_i\in \{O_n:n\in \omega\}$ such that $O_i\subseteq W$. Observing that $$p^{-1}(\pi_{i}^{-1}(V_{O_i}) )\cap H=p^{-1}_{O_i}(V_{O_i})\cap H\subseteq h^{-1}(O_i), $$ where $\pi_{i}:\prod_{k<\omega}F_{O_k}\rightarrow F_{O_i}$ is the projection of the $i$th factor. Since $h=j\circ p|_H$, we have that $j(\pi_{i}^{-1}(V_{O_i})\cap p(H))\subseteq O_i\subseteq W$. Clearly, $\pi_{i}^{-1}(V_{O_i})\cap p(H)$ is a open neighbourhood of the identity in $p(H)$, and therefore $j$ is a continuous homomorphism.

 Since $f=g\circ h$ and $h=j\circ p|_H$, we have that $f=g\circ j\circ p|_H$. Put $A=j^{-1}(g^{-1}(0)$, Then $$Z=f^{-1}(0)=p|_H^{-1}(A)=p|_H^{-1}(\overline{A}\cap p(H))=p^{-1}(\overline{A})\cap H.$$

 Since $p(G)$ is a metrizable space and $\overline{A}$ is a closed set in $p(G)$, $\overline{A}$ is a zero set in $p(G)$. Hence, $p^{-1}(\overline{A})$ is a zero set in $G$. The proof is finished.
\end{proof}

\begin{theorem}\label{TT}
Let $G$ be a regular $\mathbb{R}$-factorizable paratopological group. Then every subgroup $H$ of $G$ is $\mathbb{R}$-factorizable  if and only if $H$ is $z$-embedded in $G$.
\end{theorem}

\begin{proof}
  Since a $z$-embedded subgroup of an $\mathbb{R}_i$-factorizable paratopological group is $\mathbb{R}_i$-factorizable, for $i = 1, 2, 3$
  \cite[Theorem 3.12]{ST}, by Theorem \ref{T1} the sufficiency is proved.

Necessity. Since every regular paratopological group is completely regular \cite[Corrollary 5]{BR} and every completely regular $\mathbb{R}$-factorizable paratopological group is totally $\omega$-narrow \cite[Proposition 3.10]{XTS}, the statement directly follows from Proposition \ref{P}.
\end{proof}

It is clear that every retract of a space $X$ is $z$-embedded in $X$.  Note also
that if $G$ is a paratopological group and $H$ is an open subgroup of $G$, then $H$ is a retract of $G$. Indeed, in every left coset $U$ of $H$ in $G$, pick a point $x_U\in U$. Define $r:G\rightarrow H$ in the following way:  if $g\in H$, then $r(g)=g$; if $g\in U$ and $U\neq H$, then $r(g)=x_U^{-1}g$. Since the left cosets are open and disjoint, the continuity of $r$ is immediate. From these two observations we deduce the following result

\begin{corollary}
Let $G$ be a $\R$-factorizable paratopological group and $H$ is retract of $G$. Then $H$ is $\R$-factorizable.
\end{corollary}

\begin{corollary}
Every open subgroup of $\R$-factorizable paratopological group is $\R$-factorizable.
\end{corollary}


\end{document}